\documentclass[conference]{IEEEtran}
\IEEEoverridecommandlockouts

\usepackage{amsmath,graphicx}

\usepackage{graphicx}
\usepackage{verbatim}
\usepackage{mwe}

\usepackage{color,soul}

\usepackage{amsthm}
\usepackage{amssymb}
\usepackage{amsfonts}
\usepackage{bbm}
\usepackage{color}
\usepackage{float}
\usepackage{cite}
\usepackage{amsmath,amssymb,amsfonts}
\usepackage{algorithmic}
\usepackage{graphicx}
\usepackage{textcomp}
\usepackage{xcolor}
\def\BibTeX{{\rm B\kern-.05em{\sc i\kern-.025em b}\kern-.08em
    T\kern-.1667em\lower.7ex\hbox{E}\kern-.125emX}}

\usepackage{caption}
\usepackage{subcaption}
\makeatletter
\newcommand{\distas}[1]{\mathbin{\overset{#1}{\kern\z@\sim}}}%
\newsavebox{\mybox}\newsavebox{\mysim}
\newcommand{\distras}[1]{%
  \savebox{\mybox}{\hbox{\kern3pt$\scriptstyle#1$\kern3pt}}%
  \savebox{\mysim}{\hbox{$\sim$}}%
  \mathbin{\overset{#1}{\kern\z@\resizebox{\wd\mybox}{\ht\mysim}{$\sim$}}}%
}
\makeatother

\usepackage[ruled,vlined]{algorithm2e}

\newtheorem{theorem}{Theorem}
\newtheorem{lemma}{Lemma}
\newtheorem{definition}{Definition}
\newtheorem{assumption}{Assumption}
\newtheorem*{remark}{Remark}

\DeclareMathOperator*{\argmin}{arg\,min}

\usepackage{balance}
\usepackage{cite}

\title{Moreau Envelope ADMM for Decentralized Weakly Convex Optimization}
 \author{Reza Mirzaeifard$^\dagger$, Naveen K. D. Venkategowda$^\S$, Alexander Jung$^\star$, Stefan Werner$^\dagger$ \thanks{This work was supported by the Research Council of Norway.}\\
 $^\dagger$Department of Electronic Systems, Norwegian University of Science and Technology, Norway\\$^\S$Department of Science and Technology,  Linköping University, Sweden\\ $^\star$Department of Computer Science, Aalto University, Finland
   \\E-mails: \{reza.mirzaeifard, stefan.werner\}@ntnu.no, naveen.venkategowda@liu.se, alex.jung@aalto.fi
 }
   
\begin{document}
\maketitle
\begin{abstract}
This paper proposes a proximal variant of the alternating direction method of multipliers (ADMM) for distributed optimization. Although the current versions of ADMM algorithm provide promising numerical results  in producing solutions that are close to optimal for many convex and non-convex optimization problems, it remains unclear if they can converge to a stationary point for weakly convex and locally non-smooth functions. Through our analysis using the Moreau envelope function, we demonstrate that MADM can indeed converge to a stationary point under mild conditions. Our analysis also includes computing the bounds on the amount of change in the dual variable update step by relating the gradient of the Moreau envelope function to the proximal function. Furthermore, the results of our numerical experiments indicate that our method is faster and more robust than widely-used approaches. 
\end{abstract}
\section{Introduction}
\label{sec:intro}
Many systems, like the internet-of-things (IoT) and cyber-physical systems, comprise distributed devices and sensors that gather data for inference and decision-making. Building distributed models in such systems without data transfer to a central hub calls for distributed optimization methods involving peer-to-peer interactions. In addition, these methods allow for coping with resource constraints, e.g.,  computational resources, battery power, communication bandwidth, and privacy protection \cite{gogineni2022communication,venkategowda2020privacy,wang2019privacy}.

There is a large body of work on distributed optimization methods from different perspectives. The most direct approach to the design of distribution optimization methods is via message-passing implementations of subgradient computation within subgradient descent methods \cite{nedic2009distributed,nedic2017achieving,makhdoumi2017convergence}. Gradient methods are generalized to solve structured optimization problems using proximal methods \cite{parikh2014proximal,boyd2011distributed}. Additionally, variational methods for probabilistic models lend themselves naturally to optimization algorithms, such as variants of belief propagation \cite{wainwright2008graphical}. The subgradient method is well-known for its ease of implementation, wherein a subgradient is taken at each step, followed by an average with neighbors. On the downside, subgradient descent has a sublinear convergence rate and requires tuning of the step size. Meanwhile, ADMM performs fast and accurately in many practical convex and non-convex optimization problems. For convex objective functions, subgradient methods and ADMM are guaranteed to converge to a global optimum under suitable parameter choices \cite{makhdoumi2017convergence,lobel2010distributed,xin2020decentralized}. However, the analysis of these methods for non-convex problems  is challenging due to the potential ill-behavior of the objective function. The convergence analysis for ADMM in non-convex problems is particularly more challenging as it requires analyzing the convergence of multiple sub-problems with different structures and assumptions.

One important family of non-convex optimization problems is weakly convex problems. Several such problems arise in machine learning, including robust phase retrieval \cite{qian2017inexact}, blind deconvolution \cite{levin2011understanding}, biconvex compressive sensing \cite{ling2015self}, and dictionary learning \cite{davis2019stochastic}. Smooth functions, or functions with Lipschitz continuous gradients, are weakly convex functions.  Several non-convex optimization algorithms are proposed based on the smoothness assumption (e.g., \cite{yi2022sublinear,hong2017prox}); however, weakly convex functions are not restricted to smooth functions, and non-smooth functions can also be weakly convex \cite{duchi2019solving}. Existing work on distributed optimization of weakly convex functions includes \cite{chen2021distributed}. However, in the subgradient-based algorithm in \cite{chen2021distributed}, the local functions must satisfy the sharpness assumption, and the accuracy of the estimation depends on the step size. Although ADMM is a powerful algorithm applicable to many problems, it is not currently used to solve weakly convex problems. As seen in \cite{wang2019global, hong2016convergence, yashtini2022convergence, themelis2020douglas}, several ADMM-based works study non-convex optimization; however, these works require a smooth objective function. It is still necessary to provide a distributed ADMM-based algorithm that could work in the weakly convex setting without having any Lipschitz differentiability condition. 

We propose a Moreau envelope-based ADMM (MADM), suitable for distributed optimization where local objectives are weakly convex and not necessarily smooth. We chose the Moreau envelope-based ADMM approach because it allows us to guarantee a decrease in each primal update step and bound the amount of change in the dual update step by primary variables. This is achieved by incorporating the relationship between the Moreau envelope function and proximal function. Therefore, by selecting appropriate penalty parameters under mild conditions, including weakly convexity of each local function, a connected network, and the boundness of augmented Lagrangian function, we can ensure that the algorithm converges to a stationary point. The Moreau envelope-based ADMM approach stands out from other penalty-based ADMM algorithms due to its superior theoretical convergence properties. We conduct illustrative numerical experiments to verify the convergence properties of the proposed method. The experiments demonstrate the robustness of the proposed algorithm when we fix the penalty parameters and step size, and the problem structure remains the same. Unlike subgradient-based methods, the MADM approach ensures faster and more reliable convergence in this setting.

\noindent\textit{\textbf{Mathematical Notations}}:  
Scalars, column vectors, and matrices are respectively denoted by lowercase, bold lowercase, and bold uppercase letters. The operator $(\cdot)^\text{T}$ denotes transpose of a matrix, and the $j$th column of  matrix $\mathbf{A}$ is denoted by $\mathbf{A}_{j}$. The set $\{1,\; \cdots,\; L\}$ is denoted by $[L]$. For a function $h:\mathbb{R}^n \rightarrow \mathbb{R}$ and penalty parameter $\gamma>0$, $\mathcal{M}_{h}(\mathbf{w};\gamma)= \min_{\mathbf{x}} \big\{h(\mathbf{x})+\frac{1}{2\gamma}\|\mathbf{x}-\mathbf{w}\|_2^2 \big\}$ is the Moreau envelope function \cite{moreau1965proximite}, and  $\text{Prox}_{h}(\mathbf{w};\gamma)= \argmin_\mathbf{x} \big \{h(\mathbf{x})+\frac{1}{2\gamma}\|\mathbf{x}-\mathbf{w}\|^2\big \}$ is its associated proximal operator. 
\section{Problem formulation}
Suppose $L$ agents solve the following problem:
\begin{equation}\label{eq:1}
    \min_{\mathbf{x}} \sum_{i=1}^L f_i(\mathbf{x}),
\end{equation}
where  $f_i(\cdot):\mathbb{R}^N\rightarrow \mathbb{R}$ represents the local objective function that is only known to agent $i$. Additionally, each agent may exchange information with its neighbors through the underlying undirected communication network, which can be modeled as a graph $\mathcal{G}=(\mathcal{V},\mathcal{E})$, where $\mathcal{V}=[L]$ represents the set of agents and $\mathcal{E} \subset \mathcal{V} \times \mathcal{V}$ represents the set of edges.  In other words, the existence of $e_{i,j}\in \mathcal{E}$ indicates that $i$ and $j$ can exchange information. Due to the fact that both $e_{i,j}$ and $e_{j,i}$ denote the same edge, we merely use the expressions $e_{i,j}$ (if $i<j$) or $e_{j,i}$ (if $j < i$) to avoid repetition. Additionally, $E=|\mathcal{E}|$ is the total number of edges, and $|\mathcal{N}_i|$ is the number of neighbors for node $i$ in which $\mathcal{N}_i$ is its set of neighbors. 
In order to apply ADMM, one can rewrite \eqref{eq:1} in the form of an edge consensus problem as follows:
\begin{alignat}{2}\label{eq:2}
&
\!\min_{\{\mathbf{x}_1,\cdots,\mathbf{x}_L,\mathbf{Z}\}} &\qquad& \sum_{i=1}^{L} f_i(\mathbf{x}_i)+g(\mathbf{Z})\\ \noindent \nonumber
&\text{subject to}  &     & \mathbf{x}_i= \mathbf{z}_{i,j}, 
 \mathbf{x}_j=\mathbf{z}_{i,j}, \quad \forall e_{i,j}\in\mathcal{E}
\end{alignat}
where each $\mathbf{Z}=\{\{\mathbf{z}_{i,j}\}_{j\in \mathcal{N}_i,j>i}\}_{i=1}^{L}$ are  auxiliary variables, and $g(\cdot)=0$. The augmented Lagrangian of \eqref{eq:2} is:
\begin{multline}
      \mathcal{L}_{\rho_{\lambda}}(\mathbf{X},\mathbf{Z},\boldsymbol{\lambda})=\sum_{i=1}^{L} f_i(\mathbf{x}_i)+ \sum_{e_{i,j}\in \mathcal{E}}\Big( (\boldsymbol{\lambda}_{i,j}^{i})^{\text{T}}(\mathbf{x}_i-\mathbf{z}_{i,j})\\
 +(\boldsymbol{\lambda}_{i,j}^{j})^{\text{T}}(\mathbf{x}_j-\mathbf{z}_{i,j})+\frac{\rho_{\lambda}}{2}\|\mathbf{x}_j-\mathbf{z}_{i,j}\|^2+\frac{\rho_{\lambda}}{2}\|\mathbf{x}_i-\mathbf{z}_{i,j}\|^2 \Big),
\end{multline}
where $\mathbf{X}=[\mathbf{x}_1,\cdots,\mathbf{x}_L]$,   $\boldsymbol{\lambda}=\{\{\boldsymbol{\lambda}_{i,j}^{i},\boldsymbol{\lambda}_{i,j}^{j}\}_{j\in \mathcal{N}_i,j>i}\}_{i=1}^{L}$  are dual variables, and $\rho_{\lambda}$ is a penalty parameter.  Distributed ADMMs are iterative procedures that involve three steps at each iteration. The first step is to minimize $\mathcal{L}_{\rho_{\lambda}}$ with respect to $\mathbf{X}$. Afterward, $\mathcal{L}_{\rho_{\lambda}}$ is minimized based on $\mathbf{Z}$. In the last step, a dual gradient-ascent iteration is used to update $\boldsymbol{\lambda}$. 

\begin{definition}
A function $f(x)$ is $\rho-$weakly convex $(\rho > 0)$ if there exists a convex function $h(x)$ such that $h(x)= f(x)+\rho \|x\|^2$.
\end{definition}
Weakly convex local functions pose a challenge to distributed ADMM convergence because it can both be non-convex and non-smooth. In the absence of Lipschitz differentiability, which is coming from smoothness, and convexity of the objective function, existing distributed ADMM-based approaches cannot guarantee convergence.  In the following section, we present
an ADMM-based algorithm can deal with weakly convex functions, regardless of whether they are smooth.

\section{Moreau envelope ADMM}
The proximal augmented Lagrangian of \eqref{eq:2} can be derived as:
\begin{equation}\label{eq:prox_aug}
  \Psi_{\rho_{\lambda},{\rho_{\beta}}}\left(\mathbf{X},\mathbf{Z},\boldsymbol{\beta}, \boldsymbol{\lambda}\right) = \mathcal{L}_{\rho_{\lambda}}\left(\mathbf{X},\mathbf{Z},\boldsymbol{\lambda}\right)+\frac{\rho_{\beta}}{2} \|\mathbf{Z}-\boldsymbol{\beta}\|_F^2,
\end{equation}
where $\boldsymbol{\beta}=\{\{\boldsymbol{\beta}_{i,j}\}_{j\in \mathcal{N}_i,j>i}\}_{i=1}^{L}$ are auxiliary variables, and $\rho_{\beta}$ is a penalty parameter. In \eqref{eq:prox_aug}, The proximal term plays a crucial role in obtaining convergent results.  It regulates the behavior of the algorithm in both the $\mathbf{Z}$-update step and indirectly in the $\boldsymbol{\lambda}$-update step by incorporating the Moreau envelope function, resulting in provable convergence. According to our proposed proximal ADMM algorithm, the $(k+1)$th iteration is as follows:
\begin{subequations}
 \begin{equation}
    \mathbf{X}^{(k+1)}=
     \argmin_{\mathbf{x}}  \Psi_{\rho_{\lambda},{\rho_{\beta}}}\left(\mathbf{X},\mathbf{Z}^{(k)},\boldsymbol{\beta}^{(k)}, \boldsymbol{\lambda}^{(k)}\right),
\end{equation}
\begin{equation}
    \mathbf{Z}^{(k+1)}=
     \argmin_{\mathbf{Z}}  \Psi_{\rho_{\lambda},{\rho_{\beta}}}\left(\mathbf{X}^{(k+1)},\mathbf{Z},\boldsymbol{\beta}^{(k)}, \boldsymbol{\lambda}^{(k)}\right),
\end{equation}
\begin{equation}\
    \boldsymbol{\beta}^{(k+1)}=\boldsymbol{\beta}^{(k)}-\eta\left(\boldsymbol{\beta}^{(k)}-\boldsymbol{Z}^{(k+1)}\right)
     ,
\end{equation}
\begin{equation}
    \boldsymbol{\lambda}_{i,j}^{i,(k+1)}=\boldsymbol{\lambda}_{i,j}^{i,(k)}+\rho_{\lambda}\left(\mathbf{x}_i^{(k+1)}-\mathbf{z}_{i,j}^{(k+1)}\right), \forall e_{i,j} \in \mathcal{E}
\end{equation}
\begin{equation}
    \boldsymbol{\lambda}_{i,j}^{j,(k+1)}=\boldsymbol{\lambda}_{i,j}^{j,(k)}+\rho_{\lambda}\left(\mathbf{x}_j^{(k+1)}-\mathbf{z}_{i,j}^{(k+1)}\right), \forall e_{i,j} \in \mathcal{E}
\end{equation}
\end{subequations}
 where $\eta \in (0,2)$. 
 
More precisely, each $\mathbf{x}_i$ can be updated individually in  update-step $\mathbf{X}$, which after several simplifications, can be stated as follows:
 \begin{multline}
 \mathbf{x}_i^{(k+1)}=
 \text{Prox}_{f_i}\bigg(\frac{\sum_{j\in \mathcal{N}_i,j>i}\mathbf{z}_{i,j}^{(k)}-\frac{\lambda_{i,j}^{i,(k)}}{\rho_{\lambda}}}{|\mathcal{N}_i|}\\+ \frac{\sum_{j\in \mathcal{N}_i,j<i}\mathbf{z}_{j,i}^{(k)}-\frac{\lambda_{j,i}^{i,(k)}}{\rho_{\lambda}}}{|\mathcal{N}_i|};\frac{1}{\rho_{\lambda}|\mathcal{N}_i|}\bigg)
 \end{multline}
 Also, in the $\mathbf{Z}$-update step, $\mathbf{z}_{i,j}$ can be updated separately as:
\begin{multline}
   \mathbf{z}_{i,j}^{(k+1)}= \argmin_{\mathbf{z}_{i,j}} \bigg( (\lambda_{i,j}^{i,(k)})^{\text{T}}(\mathbf{x}_i^{(k+1)}-\mathbf{z}_{i,j})
   \\
   +(\lambda_{i,j}^{j,(k)})^{\text{T}}(\mathbf{x}_j^{(k+1)}
      -\mathbf{z}_{i,j})+\frac{\rho_{\lambda}}{2}\|\mathbf{x}_j^{(k+1)}-\mathbf{z}_{i,j}\|^2\\+\frac{\rho_{\lambda}}{2}\|\mathbf{x}_i^{(k+1)}-\mathbf{z}_{i,j}\|^2+\frac{\rho_{\beta}}{2}\|\mathbf{z}_{i,j}-\boldsymbol{\beta}_{i,j}^{(k)}\|^2\bigg), 
\end{multline}
which can be simplified as follows:
\begin{equation}
   \mathbf{z}_{i,j}^{(k+1)} =\frac{\rho_{\lambda}\left(\mathbf{x}_j^{(k+1)}+\mathbf{x}_i^{(k+1)}\right)+\rho_{\beta}\boldsymbol{\beta}_{i,j}^{(k)}+\lambda_{i,j}^{i,(k)}+\lambda_{i,j}^{j,(k)}}{2\rho_{\lambda}+\rho_{\beta}}.
\end{equation}
Moreover, for each $\boldsymbol{\beta}_{i,j}$ we have:
\begin{equation}
    \boldsymbol{\beta}_{i,j}^{(k+1)}= \boldsymbol{\beta}_{i,j}^{(k)}-\eta\left(\boldsymbol{\beta}_{i,j}^{(k)}-\boldsymbol{z}_{i,j}^{(k+1)}\right).
\end{equation}

By introducing $\boldsymbol{\lambda}_{i,j}^{i}$, $\boldsymbol{\lambda}_{i,j}^{j}$, $\mathbf{z}_{i,j}$, and $\boldsymbol{\beta}_{i,j}$ to represent $\boldsymbol{\lambda}_{j,i}^{i}$, $\boldsymbol{\lambda}_{j,i}^{j}$, $\mathbf{z}_{j,i}$, and $\boldsymbol{\beta}_{j,i}$, respectively in each agent $i$, the proposed method is simplified and summarized in Algorithm \ref{alg:1}.
\begin{algorithm}[t]
 \caption{Moreau envelope ADMM for distributed optimization (MADM)}
 \label{alg:1}
\SetAlgoLined
 Initialize $\mathbf{X}^{(0)}$, $\mathbf{Z}^{(0)}$, $\boldsymbol{\beta}^{(0)}$, $\boldsymbol{\lambda}^{(0)}$, $\rho_\beta$, $\rho_\lambda$ and $\eta\in (0,2)$\;
 \Repeat{the convergence
criterion is met}
{
\For{$i \in [L]$}{
  Update $\mathbf{x}_i$ as: $\mathbf{x}_i^{k+1}=\text{Prox}_{f_i}\left(\sum_{j\in \mathcal{N}_i}\frac{\mathbf{z}_{i,j}^{(k)}-\frac{\lambda_{i,j}^{i,(k)}}{\rho_{\lambda}}}{|\mathcal{N}_i|};\frac{1}{\rho_{\lambda}|\mathcal{N}_i|}\right)$\;
    }
    Each agent sends its local vector $\mathbf{x}_i^{k+1}$ to neighboring agents\;
    \For{$i \in [L]$}{
    \For{$j \in \mathcal{N}_i$}{
  Update $\mathbf{z}_{i,j}$ as:
  $\mathbf{z}_{i,j}^{(k+1)}= \frac{\rho_{\lambda}\left(\mathbf{x}_j^{(k+1)}+\mathbf{x}_i^{(k+1)}\right)+{\rho_{\beta}\boldsymbol{\beta}_{i,j}^{(k)}}+\lambda_{i,j}^{i,(k)}+\lambda_{i,j}^{j,(k)}}{2\rho_{\lambda}+\rho_{\beta}}$\;
  
  Update $\boldsymbol{\beta}_{i,j}$ as:
  $\boldsymbol{\beta}_{i,j}^{(k)}-\eta\left(\boldsymbol{\beta}_{i,j}^{(k)}-\boldsymbol{z}_{i,j}^{(k+1)}\right)$\;

  Update $\boldsymbol{\lambda}_{i,j}^{i}$ as:
  $\boldsymbol{\lambda}_{i,j}^{i,(k+1)}=\boldsymbol{\lambda}_{i,j}^{i,(k)}+\rho_{\lambda}\left(\mathbf{x}_i^{(k+1)}-\mathbf{z}_{i,j}^{(k+1)}\right)$\;
   Update $\boldsymbol{\lambda}_{i,j}^{j}$ as:
  $\boldsymbol{\lambda}_{i,j}^{j,(k+1)}=\boldsymbol{\lambda}_{i,j}^{j,(k)}+\rho_{\lambda}\left(\mathbf{x}_j^{(k+1)}-\mathbf{z}_{i,j}^{(k+1)}\right)$\;
  }
    }
 }\end{algorithm}

\section{Convergence proof}
This section presents the convergence analysis for Algorithm \ref{alg:1}. Several conventional assumptions are made to build our convergence proof.
\begin{assumption}\label{ass0}
The undirected graph $\mathcal{G}$ is  connected.
\end{assumption}
\begin{assumption}\label{ass1}
$ \Psi_{\rho_{\lambda},{\rho_{\beta}}}\left(\mathbf{X}^{(k)},\mathbf{Z}^{(k)},\boldsymbol{\beta}^{(k)}, \boldsymbol{\lambda}^{(k)}\right)$ is lower bounded, and  $\left(\mathbf{X}^{(k)},\mathbf{Z}^{(k)}, \boldsymbol{\beta}^{(k)}, \boldsymbol{\lambda}^{(k)}\right)$ are bounded, in each iteration $k$.
\end{assumption}
\begin{remark}
It can be shown that for coercive functions\footnote{A function $f(\cdot)$ is coercive if $f(\mathbf{x})\mathbf{x}\rightarrow \infty$ as $\|\mathbf{x}\|\rightarrow\infty$} Assumption \ref{ass1} holds true.
\end{remark}
\begin{assumption}\label{ass2}
 Local objectives $f_i(\cdot), \forall i \in [L]$, are continuous, and weakly convex by parameter $\rho_f$.
\end{assumption}
The proof of convergence relies on a canonical methodology as described in \cite[Theorem 2.9]{attouch2013convergence}. Each algorithm iteration has only one increasing step, which is the $\boldsymbol{\lambda}$-update step. As the gradient of the Moreau envelope is related to the proximal function, the amount of increase in the $\boldsymbol{\lambda}$-update step ($\rho_{\lambda}^{-1}\|\boldsymbol{\lambda}^{(k+1)}-\boldsymbol{\lambda}^{(k)}\|^2$), is bounded based on the primal and auxiliary variables. Thus, tuning the parameters lets us guarantee the {\em sufficient decrease condition} of \cite[Theorem 2.9]{attouch2013convergence}. In addition, the subgradient of the proximal augmented Lagrangian based on each of its inputs can easily be shown to be bound in each iteration, which is sufficient to validate the {\em bounded subgradient condition} of the \cite[Theorem 2.9]{attouch2013convergence}. Finally, by having the boundedness assumption and knowing that the proximal augmented Lagrangian is continuous based on each of its inputs, it can be shown that the {\em continuity condition} of \cite[Theorem 2.9]{attouch2013convergence} holds.
 \begin{figure*}[ht!]
\begin{minipage}[ht!]{0.5\textwidth}
     \centering
     \includegraphics[width=0.9\textwidth]{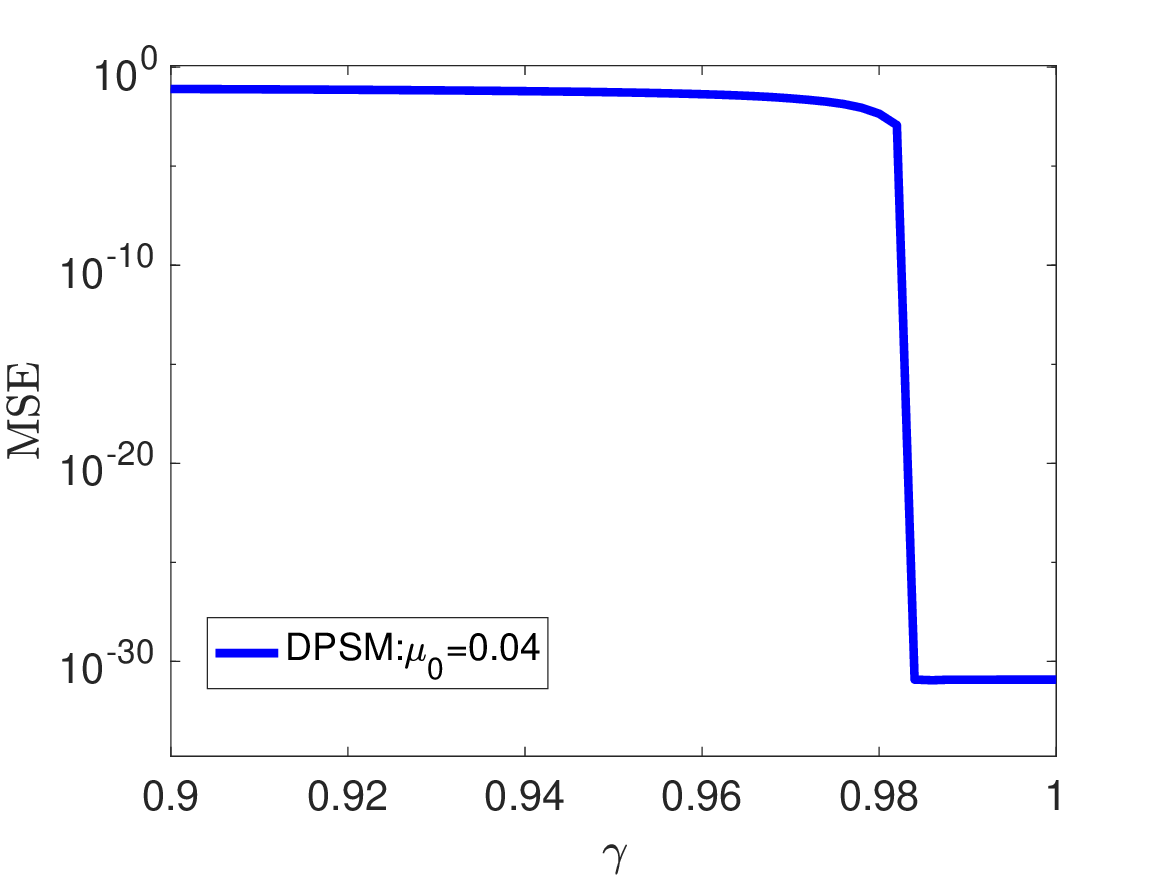}
      \caption{MSE versus $\gamma$}
 \label{fig1}
 \end{minipage}
  \begin{minipage}[ht!]{0.5\textwidth}
     \centering
     \includegraphics[width=0.9\textwidth]{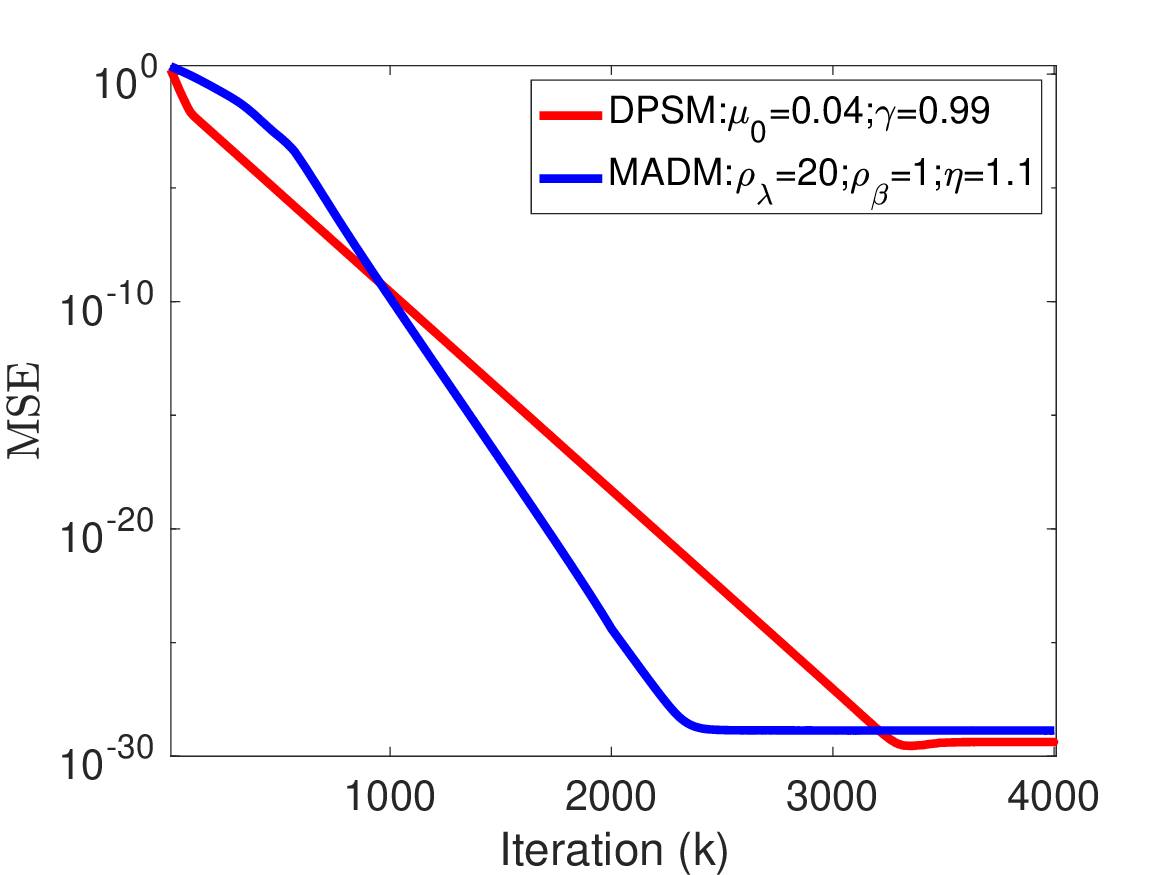}
     \caption{MSE versus iteration}
     \label{fig2}
      \end{minipage}
 \end{figure*}
 
   \begin{lemma}\label{lemma1}
  Function $g(\cdot)$, $\forall \mathbf{u},\mathbf{v} \in \mathbb{R}^{n}$, satisfies  condition:
    \begin{equation}
        \|\nabla \mathcal{M}_{g(\cdot)}(\mathbf{u},\gamma)
        -\nabla \mathcal{M}_{g(\cdot)}(\mathbf{v},\gamma)\|=0.
            \end{equation}
\end{lemma}
\begin{lemma}\label{lemma2} If Assumption \ref{ass0} is held, for any $m\geq 1$, the following inequality is held:
  \begin{multline}
 \|\boldsymbol{\lambda}^{(k+1)}-\boldsymbol{\lambda}^{(k)}\|_F^2
 \leq \\    \rho_{\beta}^{2}
\left(\|\mathbf{Z}^{(k+1)}-\mathbf{Z}^{(k)}\|_F^2+\|\boldsymbol{\beta}^{(k)}-\boldsymbol{\beta}^{(k-1)}\|_F^2\right)
 \end{multline}
 \end{lemma}
 
 \begin{proof}
  The lemma is proved by combining Lemma \ref{lemma1} with \cite[Lemma 4]{zeng2022moreau}. \end{proof}
 \begin{lemma}\label{lemma3}
  Assuming Assumptions \ref{ass0}, \ref{ass1}, and \ref{ass2} and $\rho_{\lambda}|\mathcal{N}_i|>\rho_{f}, \forall i \in [L]$, the following inequality holds:
   \begin{subequations}
 \begin{multline}
  \Psi_{\rho_{\lambda},{\rho_{\beta}}}\left(\mathbf{X}^{(k)},\mathbf{Z}^{(k)},\boldsymbol{\beta}^{(k)}, \boldsymbol{\lambda}^{(k)}\right)-\\ \Psi_{\rho_{\lambda},{\rho_{\beta}}}\left(\mathbf{X}^{(k+1)},\mathbf{Z}^{(k+1)},\boldsymbol{\beta}^{(k+1)},\boldsymbol{\lambda}^{(k+1)}\right) \geq
\end{multline}
  \begin{equation}\label{20b}
  C(\rho_{\lambda})\|\mathbf{X}^{(k+1)}-\mathbf{X}^{(k)}\|_F^2 
   +\left(\frac{\rho_{\lambda}}{2}+\frac{\rho_{\beta}}{2}\right)\|\mathbf{Z}^{(k+1)}-\mathbf{Z}^{(k)}\|_F^2 
  \end{equation}
  \begin{equation}\label{20d}
   +\frac{\rho_{\beta}}{2}\left(\frac{2}{\eta}-1\right)\|\boldsymbol{\beta}^{(k+1)}
  -\boldsymbol{\beta}^{(k)}\|_F^2 \quad 
   \end{equation}
    \begin{equation}\label{20e}
  - \rho_{\lambda}^{-1}   \rho_{\beta}^{2}
\left(\|\mathbf{Z}^{(k+1)}-\mathbf{Z}^{(k)}\|_F^2+\|\boldsymbol{\beta}^{(k)}-\boldsymbol{\beta}^{(k-1)}\|_F^2\right),    \end{equation}
 \end{subequations}
 where $C(\cdot)$ is a function with positive value for $\rho_{\lambda}$.
 \end{lemma}
 \begin{proof}
 Equation \eqref{20b} is derived based on the weak convexity of local functions, and $g(\cdot)$, while  \eqref{20d} is the result of expanding the amount of change from the $\boldsymbol{\beta}$-update step. Further, \eqref{20e} derives from multiplying the bound obtained from Lemma \ref{lemma2} with $-\rho_{\lambda}^{-1}$, which gives an upper bound for the amount of change in the $\boldsymbol{\lambda}$-update step. \end{proof}
\begin{theorem}\label{theo1}
By having $\rho_{\lambda}|\mathcal{N}_i|>\rho_{f}, \forall i \in [L]$, $\frac{1}{\eta}\geq \frac{1}{2}+  \rho_{\lambda}^{-1} \rho_{\beta}$, and $\rho_{\lambda} \geq   \frac{(2\sqrt{2}-1)}{2}\rho_{\beta}$, if Assumptions \ref{ass0},\ref{ass1}, and \ref{ass2} hold, the algorithm \ref{alg:1} converges to a stationary point. 
\end{theorem}
\begin{proof}
The {\em sufficient decrease condition} of \cite[Theorem 2.9]{attouch2013convergence} holds when $\rho_{\lambda}$, $\rho_{\beta}$, and $\rho_{\beta}$ satisfy the condition of Theorem \ref{theo1}, by Lemma \ref{lemma3}. The same results are obtained for the {\em bounded subgradient condition} of \cite[Theorem 2.9]{attouch2013convergence} when it depends on the norm of the successive difference of the variables. Finally, employing Assumption \ref{ass1} and knowing that the proximal augmented Lagrangian is continuous for each of its inputs, we can prove the {\em continuity condition} of \cite[Theorem 2.9]{attouch2013convergence}, which completes  the proof.
\end{proof}
\section{Simulation results}
In this section, we evaluate the performance of the MADM by conducting simulations of distributed robust phase retrieval with the objective function:
\begin{equation}
    \hat{\mathbf{x}}=\min_{\mathbf{x}} \frac{1}{m} \sum_{i=1}^L |\langle \mathbf{a}_{i}, \mathbf{x}\rangle^2 - b_{i}^2|,
\end{equation}
where $\mathbf{x}$ is the target signal, $\mathbf{a}_{i}$ is the  measurement and $b_{i}$ is the observation in node $i$. We assume that each node $i$ has one measurement and observation, $ \mathbf{a}_i \in \mathbb{R}^N \distas{\text{i.i.d}} \mathcal{N}(\mathbf{0},\mathbf{1}), \forall i \in [L]$, and $ \mathbf{x} \in \mathbb{R}^N \sim \mathcal{N}(\mathbf{0},\mathbf{1})$. For simplicity, we assume a noiseless setting with $b_i =\langle \mathbf{a}_{i}, \mathbf{x}\rangle ,\forall i \in [L]$. All simulations are performed by averaging over $50$ trials, and in each case, an Erdös-Rényi graph consisting of $L=50$ nodes was generated as the communication network. To evaluate the performance of this method, we also simulated the distributed projected subgradient method (DPSM) proposed in \cite{chen2021distributed}.  The mean square error $
\text{(MSE)}:=\frac{\sum_{i=1}^L||\mathbf{\hat{x}}_i-\mathbf{x}||_2^2}{L}
$ was utilized as the performance measure. Moreover,  $\mathbf{z}^{(0)}_{i,j}, \forall e_{i,j}\in \mathcal{E}$ in MADM and $\mathbf{x}^{(0)}_i,  \forall i \in \mathcal{V}$ in DPSM were initialized based on the procedure proposed in \cite[Sec. 4.2]{candes2015phase}.


We first compare the convergence rate and efficiency of the two algorithms. The dimension of the target signal was $N=10$, and for the DPSM, $\mu_0$ was $0.04$, while $\gamma$ was chosen through a grid search to achieve a minimum and stable error with fast convergence. Fig. \ref{fig1} illustrates the results of the grid search. We see that the DPSM highly depends on the choice of $\gamma$. In our algorithm we set  $\rho_{\lambda}=20$, $\rho_{\beta}=1$, and $\eta=1.1$.  These values satisfy the conditions in Theorem \ref{theo1}. Fig. \ref{fig2} shows that MADM can achieve a faster convergence rate than DPSM while maintaining a similar MSE.

Next, we study the robustness of the algorithms as a function of the dimension of the target signal for fixed parameters; $N$ ranged from $1$ to $20$. A comparison of MADM and DPSM behavior under different dimensions is illustrated in Fig. \ref{fig3}. We see that MADM is more stable than DPSM when parameters are fixed. Although Fig. \ref{fig1} indicates $\gamma=0.99$ is in the safe zone for $N=10$, it fails for $N=8$ and $N>10$.
\begin{figure}[t!]
      \centering
      \includegraphics[width=0.45\textwidth]{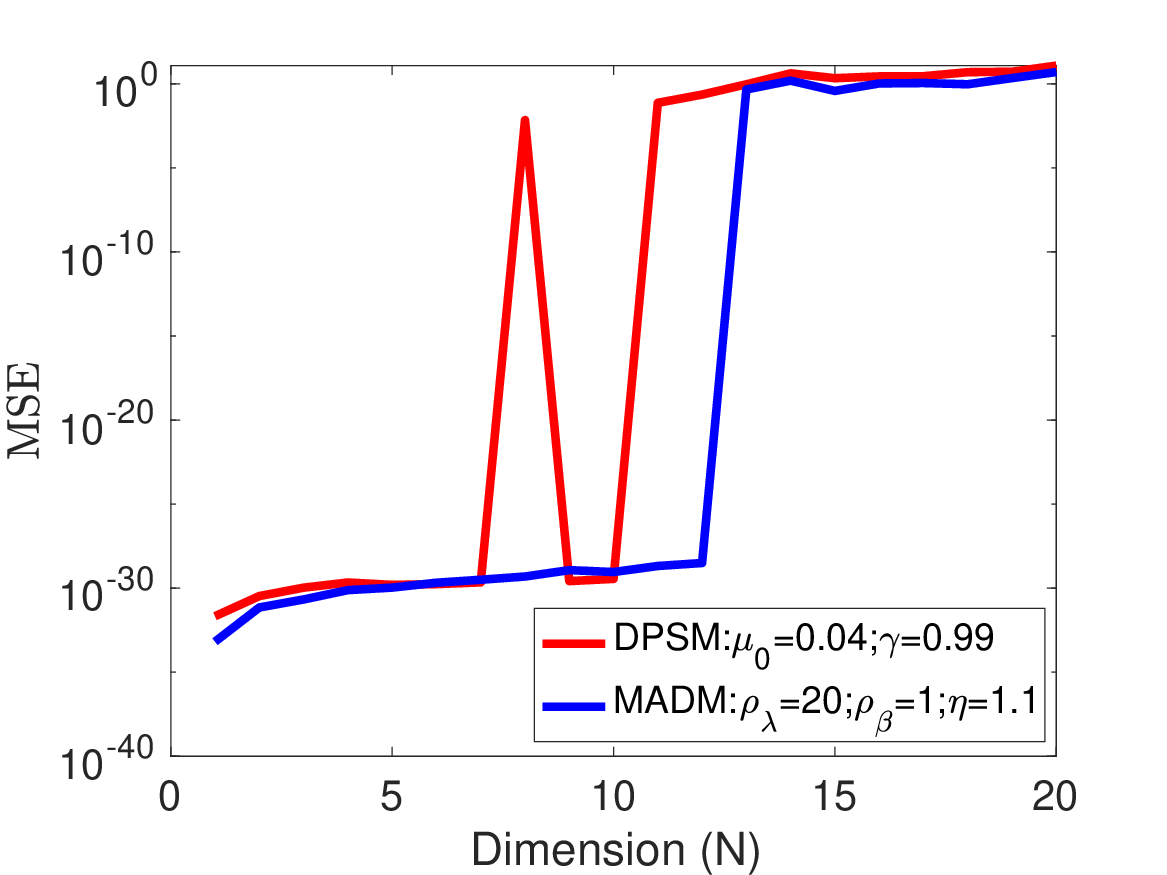}
    \caption{MSE versus dimension}
     \label{fig3}
 \end{figure}
\section{Conclusions}
This paper presented a new proximal variant of the ADMM algorithm, named MADM, for solving distributed optimization problems. Our analysis demonstrated that the proposed method could be applied to weakly convex functions under mild conditions. In particular, we derived a bound on the change in the dual variable update step by leveraging the relationship between the gradient of the Moreau envelope function and the proximal function. This allowed us to ensure convergence to a stationary point. The simulation results showed that MADM outperforms subgradient methods in terms of speed and robustness. These findings suggest that MADM can be a promising tool for solving a wide range of distributed optimization problems in practice. 

\bibliographystyle{IEEEtran}
\bibliography{strings}



\end{document}